\documentclass[preprint, 10pt]{amsart}

\usepackage{graphicx,subfigure}
\usepackage{epstopdf}
\usepackage{amsmath,amsthm,amssymb}
\usepackage{rotating}
\usepackage{morefloats}
\usepackage{comment}

\newtheorem{thm}{Theorem}
\newtheorem{crl}{Corollary}
\newtheorem{lem}{Lemma}

\begin{document}
\title{K\"ahlerian information geometry for signal processing}
\author{Jaehyung Choi}
\address{Department of Applied Mathematics and Statistics\\
SUNY, Stony Brook, NY 11794, USA}
\email{jj.jaehyung.choi@gmail.com}
\author{Andrew P. Mullhaupt}
\address{Department of Applied Mathematics and Statistics\\
SUNY, Stony Brook, NY 11794, USA}
\email{doc@zen-pharaohs.com}

\begin{abstract}
We prove the correspondence between the information geometry of a signal filter and a K\"ahler manifold. The information geometry of a minimum-phase linear system with a finite complex cepstrum norm is a K\"ahler manifold. The square of the complex cepstrum norm of the signal filter corresponds to the K\"ahler potential. The Hermitian structure of the K\"ahler manifold is explicitly emergent if and only if the impulse response function of the highest degree in $z$ is constant in model parameters. The K\"ahlerian information geometry takes advantage of more efficient calculation steps for the metric tensor and the Ricci tensor. Moreover, $\alpha$-generalization on the geometric tensors is linear in $\alpha$. It is also robust to find Bayesian predictive priors, such as superharmonic priors, because Laplace--Beltrami operators on K\"ahler manifolds are in much simpler forms than those of the non-K\"ahler manifolds. Several time series models are studied in the K\"ahlerian information geometry.
\end{abstract}

\maketitle

\section{Introduction}

Since the introduction of Riemannian geometry to statistics \cite{Rao:1945, Jeffreys:1946}, information geometry has been developed along various directions. The statistical curvature as the differential-geometric analogue of information loss and sufficiency was proposed by Efron \cite{Efron:1975}. The $\alpha$-duality of information geometry was found by Amari \cite{Amari:1990}. Not being limited to statistical inference, information geometry has become popular in many different fields, such as information-theoretic generalization of the expectation-maximization algorithm \cite{Matsuyama:2003}, hidden Markov models \cite{Matsuyama:2011}, interest rate modeling \cite{Brody:2001}, phase transition \cite{Janke:2004, Zanardi:2007} and string theory \cite{Heckman:2013}. More applications can be found in the literature \cite{Arwini:2008} and the references therein.

In particular, time series analysis and signal processing are well-known applications of information geometry. Ravishanker \textit{et al}. \cite{Ravishanker:1990p5895} found the information geometry of autoregressive moving average (ARMA) models in the coordinate system of poles and zeros. It was also extended to fractionally-integrated ARMA (ARFIMA) models \cite{Ravishanker:2001p5836}. The information geometry of autoregressive (AR) models in the reflection coefficient coordinates was also reported by Barbaresco \cite{Barbaresco:2006}. In the information-theoretic framework, Bayesian predictive priors outperforming the Jeffreys prior were derived for the AR models by Komaki \cite{Komaki:2006}.

K\"ahler manifolds are interesting topics in differential geometry. On a K\"ahler manifold, the metric tensor and the Levi--Civita connection are straightforwardly calculated from the K\"ahler potential, and the Ricci tensor is obtained from the determinant of the metric tensor. Moreover, its holonomy group is related to the unitary group. Because of these properties, many implications of K\"ahler manifolds are found in mathematics and theoretical physics. In addition to these fields, information geometry is one of those fields where the K\"ahler manifolds are intriguing. After the symplectic structure in information geometry and its connection to statistics were discovered \cite{Barndorff-Nielsen:1997}, Barbaresco \cite{Barbaresco:2006} notably introduced K\"ahler manifolds to information geometry for time series models and also generalized the differential-geometric approach with mathematical structures, such as Koszul geometry \cite{Barbaresco:2012, Barbaresco:2014}. Additionally, Zhang and Li \cite{Zhang:2013} found symplectic and K\"ahler structures in divergence functions.

In this paper, we prove that the information geometry of a signal filter with a finite complex cepstrum norm is a K\"ahler manifold. The K\"ahler potential of the geometry is the square of the Hardy norm of the logarithmic transfer function of a linear system. The Hermitian structure of the manifold is explicitly seen in the metric tensor under certain conditions on the transfer functions of linear models and filters. The calculation of geometric objects and the search for Bayesian predictive priors are simplified by exploiting the properties of K\"ahler geometry. Additionally, $\alpha $-correction terms on the geometric objects exhibit $\alpha$-linearity. This paper is structured as follows. In the next section, we shortly review information geometry for signal processing and derive basic lemmas in terms of the spectral density function and transfer function. In Section \ref{sec_kahler_theory}, main theorems for K\"ahlerian information manifolds are proven and the consequences of the theorems are provided. The implications of K\"ahler geometry to time series models are reported in Section \ref{sec_kahler_example}. We conclude the paper in Section \ref{sec_kahler_conclusion}.

\section{Information Geometry for Signal Processing}

\label{sec_kahler_background}

\subsection{Spectral Density Representation in the Frequency Domain}

We model an output signal $y( w) $ as a linear system with a transfer function $h(w;\boldsymbol{\xi })$ of model parameters $\boldsymbol{\xi }=(\xi^{1},\xi ^{2},\cdots ,\xi ^{n})$:
\begin{equation}
	y(w)=h(w;\boldsymbol{\xi })x(w)\nonumber
\end{equation}
where $x(w)$ is an input signal in frequency domain $w$. Complex inputs, outputs and model parameters are considered in this paper. The properties of a given signal filter are characterized by the transfer function $h(w;\boldsymbol{\xi })$ and the model parameters $\boldsymbol{\xi }$.

	In signal processing, one of the most important quantities is the spectral density function. The spectral density function $S(w;\boldsymbol{\xi })$ is defined as the absolute square of the transfer function: 
\begin{equation}
	\label{sdf_tf}
	S(w;\boldsymbol{\xi })=|h(w;\boldsymbol{\xi })|^{2}.
\end{equation}
The spectral density function describes the way that energy in the frequency domain is distributed by a given signal filter. In terms of signal amplitude, the spectral density function encodes an amplitude response to a monochromatic input $e^{iw}$. For example, the spectral density function of the all-pass filter is constant in the frequency domain, because the filter passes all inputs to outputs up to the phase difference regardless of frequency. The high-pass filters only allow the signals in the high-frequency domain. Meanwhile, the low-pass filters only permit low-frequency inputs. The properties of other well-known filters are also described by their specific spectral density functions. 

	The spectral density function is also important in information geometry, because the information-geometric objects of the signal processing geometry are derived from the spectral density function \cite{Amari:1987,Amari:2000}. Among the geometric objects, the length and distance concepts are most fundamental in geometry. One of the most important distance measures in information geometry is the $\alpha$-divergence, also known as Chernoff's $\alpha $-divergence, that is the only divergence which is both an $f$-divergence and a Bregman divergence \cite{Amari:2009}. The $\alpha $-divergence between two spectral density functions $S_{1}$ and $S_{2}$ is defined~as 
\begin{equation}
	D^{(\alpha )}(S_{1}||S_{2})=\left\{ 
	\begin{array}{ll}
	\frac{1}{2\pi \alpha ^{2}}\int_{-\pi }^{\pi }\big\{\Big(\frac{S_{2}}{S_{1}}\Big)^{\alpha }-1-\alpha \log {\frac{S_{2}}{S_{1}}}\big\}dw & (\alpha \neq 0)\\ 
	\frac{1}{4\pi }\int_{-\pi }^{\pi }\big(\log{S_{2}}-\log{S_{1}}\big)^{2}dw & (\alpha =0)
	\end{array}
	\right. \nonumber
\end{equation}
and the divergence conventionally measures the distance from $S_{1}$ to $S_{2}$. The $\alpha $-divergence, except for $\alpha =0$, is a pseudo-distance, because it is not symmetric under exchange between $S_1$ and $S_2$. In spite of the asymmetry, the $\alpha $-divergence is frequently used for measuring differences between two linear models or two filters. Some $\alpha$-divergences are more popular than others, because those divergences have been already known in information theory and statistics. For example, the $(-1)$-divergence is the Kullback--Leibler divergence. The 0-divergence is well known as the square of the Hellinger distance in statistics. The Hellinger distance is locally asymptotically equivalent to the information distance and globally tightly bounded by the information distance \cite{Mullhaupt:2012}.

	The metric tensor of a statistical manifold, also known as the Fisher information matrix, is derived from the $\alpha$-divergence. In order to define the information geometry of a linear system, the conditions on a signal filter are found in Amari and Nagaoka \cite{Amari:2000}: stability, minimum phase and 
	
	\begin{equation}
		\frac{1}{2\pi}\int_{-\pi}^{\pi} |\log{S(w;\boldsymbol{\xi})}|^2 dw<\infty\nonumber
	\end{equation}
	which imposes that the unweighted power cepstrum norm \cite{Bogert:1967, Martin:2000} is finite. According to the literature~\cite{Amari:1987, Amari:2000}, the metric tensor of the linear system geometry is given by

\begin{equation}
	\label{metric_sdf}
	g_{\mu\nu}(\boldsymbol{\xi })=\frac{1}{2\pi }\int_{-\pi }^{\pi }( \partial_{\mu}\log {S}) ( \partial _{\nu}\log {S}) dw
\end{equation}
where the partial derivatives are taken with respect to the model parameters $\boldsymbol{\xi }$, {\em i.e.}, $\partial_\mu=\frac{\partial}{\partial \xi^\mu}$. Since the dimension of the manifold is $n$, the metric tensor is an $n\times n$ matrix.

Other information-geometric objects are also determined by the spectral density function. The $\alpha $-connection, which encodes the change of a vector being parallel-transported along a curve, is expressed with
\begin{equation}
	\label{connection_sdf}
	\Gamma _{\mu\nu,\rho}^{(\alpha )}(\boldsymbol{\xi })=\frac{1}{2\pi }\int_{-\pi}^{\pi }(\partial _{\mu}\partial _{\nu}\log {S}-\alpha ( \partial _{\mu}\log{S}) ( \partial _{\nu}\log {S}) )(\partial_{\rho}\log {S})dw
\end{equation}
where $\alpha $ is a real number. Notice that the $\alpha $-connection is not a tensor. The $\alpha $-connection is related to the Levi--Civita connection, $\Gamma _{\mu\nu,\rho}(\boldsymbol{\xi })$, also known as the metric connection. The relation is given by the following equations: 
\begin{align}
	\Gamma _{\mu\nu,\rho}^{(\alpha )}(\boldsymbol{\xi })&=\Gamma _{\mu\nu,\rho}(\boldsymbol{\xi})-\frac{\alpha }{2}T_{\mu\nu,\rho}(\boldsymbol{\xi })
	\label{alpha_connection_t} \\
	T_{\mu\nu,\rho}(\boldsymbol{\xi })&=\frac{1}{\pi }\int_{-\pi }^{\pi }(\partial _{\mu}\log {S}) ( \partial _{\nu}\log {S}) (\partial_{\rho}\log {S})dw
\end{align}
where the tensor $T$ is symmetric under the exchange of the indices. The Levi--Civita connection corresponds to the $\alpha =0$ case.

These information-geometric objects have interesting properties with the reciprocality of spectral density functions. The spectral density function of an inverse system is the reciprocal spectral density function of the original system. The geometric properties of the inverse system are described by the $\alpha $-dual description. The following lemma shows the correspondence between the reciprocality of the spectral density function and the $\alpha$-duality.
\begin{lem}
\label{lem_sdf_inv} The information geometry of an inverse system is the $\alpha $-dual geometry to the information geometry of the original system.
\end{lem}

\begin{proof}
The metric tensor is invariant under the reciprocality of spectral density functions, {\em i.e.}, plugging $S^{-1}$ into Equation (\ref{metric_sdf}) provides the identical metric tensor. 

	Meanwhile, the $\alpha$-connection is not invariant under the reciprocality and exhibits a more interesting property. The $\alpha $-connection from the reciprocal spectral density function is given by
\begin{align}
	\Gamma _{\mu\nu,\rho}^{(\alpha )}(S^{-1};\boldsymbol{\xi }) &=\frac{1}{2\pi }\int_{-\pi }^{\pi }(\partial _{\mu}\partial _{\nu}\log {S}+\alpha (\partial_{\mu}\log {S})(\partial _{\nu}\log {S}))(\partial _{\rho}\log {S})dw \nonumber\\
	&=\Gamma _{\mu\nu,\rho}^{(-\alpha )}(S;\boldsymbol{\xi }) \nonumber
\end{align}
and the above equation shows that the $\alpha $-connection induced by the reciprocal spectral density function corresponds to the $( -\alpha ) $-connection of the original geometry. 

Similar to the $\alpha$-connection, the $\alpha$-divergence is equipped with the same property. The $\alpha $-divergence between two reciprocal spectral density functions is straightforwardly found from the definition of the $\alpha$-divergence, and it is represented by the $( -\alpha ) $-divergence between the two spectral density functions: 
\begin{eqnarray}
D^{(\alpha )}(S_{1}^{-1}||S_{2}^{-1})=D^{(-\alpha )}(S_{1}||S_{2}).\nonumber
\end{eqnarray}
	Using the inverse systems, we can construct the $\alpha$-dual description of signal processing models in information geometry. The multiplicative inverse of a spectral density function corresponds to the $\alpha$-duality of the geometry.
\end{proof}

Lemma \ref{lem_sdf_inv} indicates that given a linear system geometry, there is no way to discern whether the metric tensor is derived from the filters with $S$ or $S^{-1}$. Additionally, the model $S^{-1}$ is $( -\alpha ) $-flat if and only if $S$ is $\alpha $-flat. The 0-connection is self-dual under the reciprocality. A consequence of Lemma \ref{lem_sdf_inv} is the following multiplication rule: 
\begin{align}
	D^{(\alpha )}(S_{1}||S_{2}^{-1}) &=\frac{1}{2\pi \alpha ^{2}}\int_{-\pi
}^{\pi }\big\{(S_{1}S_{2})^{-\alpha }-1+\alpha \log {(S_{1}S_{2})}\big\}dw 
\nonumber \\
	&=D^{(-\alpha )}(S_{0}||S_{1}S_{2})=D^{(\alpha )}(S_{1}S_{2}||S_{0}) \nonumber
\end{align}
where $S_{0}$ is the unit spectral density function of the all-pass filter. Plugging $S_{1}=S_{0}$ and $S_{2}=S$, we have $D^{(0)}(S_{0}||S^{-1})=D^{(0)}(S_{0}||S)=D^{(0)}(S||S_{0})$.

	We observe that the bilateral transfer functions $\log |h(e^{iw};\boldsymbol{\xi })|^{2}\in L^{2}\left( \mathbb{T}\right) $ are isomorphically embedded in the space $\mathbb{R}\oplus zH^{2}\left( \mathbb{D}\right)$.
\begin{lem}
	\label{Wiener-Khinchin factors} 
	Let $\log \left\vert h(e^{iw};\boldsymbol{\xi })\right\vert ^{2}\in L^{2}\left( \mathbb{T}\right)$. Then, there is an analytic function $f\in $ $\exp \left( H^{2}\left( \mathbb{D}\right) \right)$, such that
\begin{equation}
	\left\vert h(e^{iw};\boldsymbol{\xi })\right\vert ^{2}=\left\vert f(e^{iw};\boldsymbol{\xi })\right\vert ^{2}\nonumber
\end{equation}
and
\begin{equation}
	\left\Vert \log \left\vert h(e^{iw};\boldsymbol{\xi })\right\vert ^{2}-\log\left\vert h(1;\boldsymbol{\xi })\right\vert ^{2}\right\Vert _{L^{2}\left( \mathbb{T}\right) }=\left\Vert \log \left\vert f(e^{iw};\boldsymbol{\xi })\right\vert^{2}-\log \left\vert f(1;\boldsymbol{\xi })\right\vert ^{2}\right\Vert _{H^{2}\left( \mathbb{D}\right) }.\nonumber
\end{equation}
This has the interpretation that the information manifold of $\log \left\vert h(e^{iw};\boldsymbol{\xi })\right\vert ^{2}\in L^{2}$ is isometric to the Hardy--Hilbert space.
\end{lem}

\begin{proof}
$\log h(e^{iw};\boldsymbol{\xi })$ is represented by the Fourier series:

\begin{equation}
	\log \left\vert h(e^{iw};\boldsymbol{\xi })\right\vert ^{2}=\sum_{r=-\infty}^{\infty }a_{r}e^{irw}\nonumber
\end{equation}
and since $\log \left\vert h(e^{iw};\boldsymbol{\xi })\right\vert ^{2}$ is real, we have $a_{-r}=\bar{a}_{r}$, and in particular, $a_{0}$ is real. We define the conjugate series by the coefficients $\tilde{a}_{r}$, so that $a_{r}+i\tilde{a}_{r}=0$ for $r<0$ and $\tilde{a}_{r}$ for $r>0$; so that $\tilde{a}\left( e^{i\theta }\right) $ is real valued, we choose $\tilde{a}_{0}=0$. This implies
\begin{equation}
	\tilde{a}_{r}=\left\{ 
	\begin{array}{ll}
	-\frac{1}{i}a_{r} & (r<0)\\ 
	\frac{1}{i}a_{r} & (r>0)
	\end{array}
	\right. \nonumber
\end{equation}
and if $\left\{ a_{r}\right\} \in l^{p}$ for $1\leq p\leq \infty $, then $\left\{ \tilde{a}_{r}\right\} \in l^{p}$, in particular,
\begin{equation}
	\sum_{r\neq 0}\left\vert a_{r}\right\vert ^{2}=\sum_{r\neq 0}\left\vert \tilde{a}_{r}\right\vert ^{2}.
\end{equation}
The analytic function $f\left( z\right) =\exp \left( a_{0}+a\left( z\right)+i\tilde{a}\left( z\right) \right) $ has
\begin{equation}
	\log \left\vert h(e^{iw};\boldsymbol{\xi })\right\vert ^{2}=\log \left\vert f(e^{iw};\boldsymbol{\xi })\right\vert ^{2}\nonumber
\end{equation}
and
\begin{equation}
	\|\log f(z;\boldsymbol{\xi })-\log f\left( 1;\boldsymbol{\xi }\right) \|_{H^{2}}^{2}=\left\Vert \log \left\vert h(e^{iw};\boldsymbol{\xi })\right\vert ^{2}-\log \left\vert h(1;\boldsymbol{\xi })\right\vert ^{2}\right\Vert^2_{L^{2}\left( \mathbb{T}\right) }<\infty \nonumber 
\end{equation}
and because $f\in \exp \left( zH^{2}\left( \mathbb{D}\right) \right) $, $f$ (and $f^{-1}$) is outer, we may write
\begin{equation}
	h(e^{iw};\boldsymbol{\xi })=u(e^{iw};\boldsymbol{\xi })f(e^{iw};\boldsymbol{\xi }) \nonumber
\end{equation}
where $\log u(e^{iw};\boldsymbol{\xi })\in L^{2}$ is pure imaginary, that is, $\left\vert u(e^{iw};\boldsymbol{\xi })\right\vert =1$.
\end{proof}

This has the interpretation that $h$ has a well-defined outer factor, and the information geometry of $h$ depends only on $h$. In the case that the power series coefficients $a_{k}\left( \boldsymbol{\xi }\right) $ are continuous, smooth, analytic, \textit{etc}., then the embedding is likewise smooth.

\subsection{Transfer Function Representation in the $z$ Domain}

By using transfer functions, it is also possible to reproduce all of the previous results with the spectral density function. With Fourier transformation and $Z$-transformation, $z=e^{iw}$, a transfer function $h(z;\boldsymbol{\xi })$ is expressed with a series expansion of $z$,
\begin{equation}
	h(z;\boldsymbol{\xi })=\sum_{r=-\infty }^{\infty }h_{r}(\boldsymbol{\xi })z^{-r}
\end{equation}
where $h_{r}(\boldsymbol{\xi })$ is an impulse response function. It is a bilateral (or two-sided) transfer function expression, which has both positive and negative degrees in $z$, including the zero-th degree. In the causal response case that $h_{r}(\boldsymbol{\xi })=0$ for all negative $r$, the transfer function is unilateral. In many applications, the main concern is the causality of linear filters, which is represented by unilateral transfer functions. In this paper, we start with bilateral transfer functions as generalization and then will focus on causal filters.

In the complex $z$-domain, all formulae for the information-geometric objects are identical to the expressions in the frequency domain, except for the change of the integral measure:
\begin{equation}
	\frac{1}{2\pi }\int_{-\pi }^{\pi }G(e^{iw};\boldsymbol{\xi })dw\rightarrow \frac{1}{2\pi i}\oint_{|z|=1}G(z;\boldsymbol{\xi })\frac{dz}{z}\nonumber
\end{equation}
for an arbitrary integrand $G$. Since the evaluation of the integration is obtained from the line integral along the unit circle on the complex plane, it is easy to calculate the above integration with the aid of the residue theorem. According to the residue theorem, the poles only inside the unit circle contribute to the value of the integration. If $G(z;\boldsymbol{\xi})$ is analytic on the unit disk, the constant term in $z$ of $G(z;\boldsymbol{\xi})$ is the value of the integration. For more details, see Cima \textit{et al}. \cite{Cima:2006} and the references therein.

One advantage of using $Z$-transformation is that a transfer function can be understood in the framework of functional analysis. A transfer function defined on the complex plane is expanded by the orthonormal basis $z^{-r}$ for integers $r$ with impulse response functions as the coefficients. In functional analysis, it is possible to define the inner product between two complex functions $F$ and $G$ in the Hilbert space:
\begin{equation}
	\langle F,G\rangle =\frac{1}{2\pi i}\oint_{|z|=1}F(z)\overline{G(z)}\frac{dz}{z}.\nonumber
\end{equation}
By using this inner product, the condition for the stationarity, $\sum_{r=0}^{\infty }|h_{r}|^{2}<\infty $, is written as the Hardy norm ($H^{2}$-norm) in complex functional analysis,
\begin{equation}
	\|h(z;\boldsymbol{\xi })\|_{H^{2}}^{2}=\langle h(z;\boldsymbol{\xi }),h(z;\boldsymbol{\xi })\rangle=\sum_{r=0}^{\infty }|h_{r}|^{2}<\infty. \nonumber
\end{equation}
Since the functional space with a finite Hardy norm is called the Hardy--Hilbert space $H^{2}$, the unilateral transfer functions satisfying the stationarity condition live on the $H^2$-space. A transfer function of a stationary system is a function in the $L^2$-space if the transfer function is in the bilateral form. 

	The conditions on the transfer function of a signal filter are also necessary for defining the information geometry of a linear system in terms of the transfer function. Similar to the spectral density representation, the conditions on the linear filters are stability and minimum phase. In addition to these conditions, we also need the following condition on the finite $H^2$-norm of the logarithmic transfer function,
	\begin{equation}
		\frac{1}{2\pi i}\oint_{|z|=1} |\log{h(z;\boldsymbol{\xi})}|^2\frac{dz}{z}<\infty\nonumber
	\end{equation}
	and for the above condition, it is also known that the unweighted complex cepstrum norm \cite{Oppenheim:1965, Martin:2000} is finite. From now on, signal filters in this paper are the linear systems satisfying the above norm conditions. This is a necessary condition for a finite power cepstrum norm.
	
	It is natural to complexify the coordinate system as being used in the complex differential geometry. In holomorphic and anti-holomorphic coordinates, the metric tensor of a linear system geometry is represented by
\begin{equation}
	g_{\mu \nu }=\frac{1}{2\pi i}\oint_{|z|=1}\partial _{\mu }\big(\log {h(z;\boldsymbol{\xi })}+\log {\bar{h}(\bar{z};\bar{\boldsymbol{\xi }})}\big)\partial _{\nu}\big(\log {h(z;\boldsymbol{\xi })}+\log {\bar{h}(\bar{z};\bar{\boldsymbol{\xi }}})\big)\frac{dz}{z}\nonumber
\end{equation}
where both $\mu $ and $\nu $ run over all holomorphic and anti-holomorphic coordinates, {\em i.e.}, $\mu ,\nu =1,2,\cdots ,n,\bar{1},\bar{2},\cdots ,\bar{n}$. 

	The components of the metric tensor are categorized into two classes: one with pure indices from holomorphic coordinates and anti-holomorphic coordinates, and another with the mixed indices. The metric tensor components in these categories are given by
\begin{eqnarray} 
	\label{metric_complex_coordinate1}
	g_{ij}(\boldsymbol{\xi})=\frac{1}{2\pi i}\oint_{|z|=1} \partial_i \log{h(z;\boldsymbol{\xi})} \partial_j \log{h(z;\boldsymbol{\xi})} \frac{dz}{z} \\
	\label{metric_complex_coordinate2}
	g_{i\bar{j}}(\boldsymbol{\xi})=\frac{1}{2\pi i}\oint_{|z|=1} \partial_i \log{h(z;\boldsymbol{\xi})} \partial_{\bar{j}} \log{\bar{h}(\bar{z};\bar{\boldsymbol{\xi}})} \frac{dz}{z}
\end{eqnarray}
where $g_{\bar{i}\bar{j}}=(g_{ij})^{*}$ and $g_{\bar{i}j}=(g_{i\bar{j}})^{*}$, and the indices $i$ and $j$ run from one to $n$. It is also possible to express the $\alpha$-connection and the $\alpha$-divergence in terms of the transfer function by using Equation (\ref{sdf_tf}), the relation between the transfer function and the spectral density function.

	It is noteworthy that the information geometry of a linear system is invariant under the multiplicative factor of $z$ in the transfer function, because the metric tensor is not changed by the factorization. The invariance is also true for the geometry induced by the spectral density function.
\begin{lem}
	\label{lem_transfer_factorization} 
	The information geometry of a signal filter is invariant under the multiplicative factor of $z$. 
\end{lem}

\begin{proof}
Any transfer function can be factored $z^R$ out in the form of 
\begin{equation}
	h(z;\boldsymbol{\xi})=z^R\tilde{h}(z;\boldsymbol{\xi}) \nonumber
\end{equation}
where $R$ is an integer and $\tilde{h}$ is the factored-out transfer function. In the spectral density function representation, the contribution of the factorization is $|z|^{2R}$, and it is a unity in the line integration. It imposes that the metric tensor, the $\alpha$-connection and the $\alpha$-divergence are independent of the factorization.

When a transfer function is considered, the same conclusion is obtained. Since the contribution from the factorization parts, $\log{z^R}$, is canceled by the partial derivatives in the metric tensor and the $\alpha$-connection expression, the geometry is invariant under the factorization. It is also easy to show that $\alpha$-divergence is also not changed by the factorization. Another explanation is that the terms of $\partial_i h/h$ in the metric tensor and the $\alpha$-connection are invariant under $z^R$-scaling.
\end{proof}

Based on Lemma \ref{lem_transfer_factorization}, it is possible to obtain the unilateral transfer function from a transfer function with a finite upper bound in degrees of $z$. In particular, this factorization invariance of the geometry is useful in the case that the transfer function has a finite number of terms in the non-causal direction of the bilateral transfer function. If the highest degree in $z$ of the transfer function is finite, the transfer function is factored out as
\begin{align}
	h(z;\boldsymbol{\xi }) &=z^{R}(h_{-R}+h_{-(R-1)}z^{-1}+\cdots) \nonumber \\
	&=z^{R}\tilde{h}(z;\boldsymbol{\xi }) \nonumber
\end{align}
where $R$ is the maximum degree in $z$ of the transfer function and $\tilde{h}$ is a unilateral transfer function.

A bilateral transfer function can be expressed with the multiplication of a unilateral transfer function $f(z;\boldsymbol{\xi })$ and an analytic function $a(z;\boldsymbol{\xi })$ on the disk: 
\begin{align}
	h(z;\boldsymbol{\xi }) &=f(z;\boldsymbol{\xi })a(z;\boldsymbol{\xi }) \nonumber \\
	&=(f_{0}+f_{1}z^{-1}+f_{2}z^{-2}+\cdots)(a_{0}+a_{1}z^{1}+a_{2}z^{2}+\cdots )\nonumber
\end{align}
where $f_{r}$ and $a_{r}$ are functions of $\boldsymbol{\xi }$. For a causal filter, all $a_{i}$'s are zero, except for $a_{0}$. This decomposition also includes the case of Lemma \ref{lem_transfer_factorization} by setting $a_{i}=0$ for $i<R$ and $a_{R}=1$. However, it is natural to take $f_{0}$ and $a_{0}$ as non-zero functions of $\boldsymbol{\xi}$. This is because powers of $z$ could be factored out for non-zero coefficient terms with the maximum degree in $f(z;\boldsymbol{\xi })$ and the minimum degree in $a(z;\boldsymbol{\xi })$, and the transfer function is reducible to 
\begin{equation}
	h(z;\boldsymbol{\xi })=z^{R}\tilde{h}(z;\boldsymbol{\xi })\nonumber
\end{equation}
where $\tilde{h}(z;\boldsymbol{\xi })$ has non-zero $\tilde{f}_{0}$ and $\tilde{a}_{0}$ and $R$ is an integer, which is the sum of the degrees in $z$ with the first non-zero coefficient terms from $f(z;\boldsymbol{\xi })$ and $a(z;\boldsymbol{\xi })$, respectively. By Lemma \ref{lem_transfer_factorization}, the information geometry of the linear system with the transfer function $h(z;\boldsymbol{\xi })$ is the same as the geometry induced by the factored-out transfer function $\tilde{h}(z;\boldsymbol{\xi })$.

The relation between $f(z;\boldsymbol{\xi}), a(z;\boldsymbol{\xi})$ and $h(z;\boldsymbol{\xi})$ is described by the following Toeplitz system: 
\begin{equation}
\begin{pmatrix}
h_0 & h_1 & h_2 &\ldots\\
h_{-1} & h_0 & h_1 & \ldots\\
h_{-2} & h_{-1} & h_0 & \ldots\\
\vdots & \vdots & \vdots & \ddots
\end{pmatrix}
=\begin{pmatrix}
f_0 & f_1 & f_2 &\ldots\\
0 & f_0 & f_1 & \ldots\\
0 & 0 & f_0 & \ldots\\
\vdots & \vdots & \vdots & \ddots
\end{pmatrix}
\begin{pmatrix}
a_0 & 0 & 0 &\ldots\\
a_1 & a_0 & 0 & \ldots\\
a_2 & a_1 & a_0 & \ldots\\
\vdots & \vdots & \vdots & \ddots
\end{pmatrix}.\nonumber
\end{equation}
For a given $h(z;\boldsymbol{\xi})$, $f_r$ is determined by the coefficients of $a(z;\boldsymbol{\xi})$, {\em i.e.}, if we choose $a(z;\boldsymbol{\xi})$, $f(z;\boldsymbol{\xi})$ is conformable to the choice under the above Toeplitz system. The following lemma is noteworthy for further discussions. It is the generalization of Lemma \ref{lem_transfer_factorization}.
\begin{lem}
\label{lem_gauge} The information geometry of a signal filter is invariant under the choice of $a(z;\boldsymbol{\xi})$.
\end{lem}

\begin{proof}
It is obvious that the information geometry of a linear system is only decided by the transfer function $h(z;\boldsymbol{\xi})$. Whatever $a(z;\boldsymbol{\xi})$ is chosen, the transfer function is the same, because $f(z;\boldsymbol{\xi})$ is conformable to the Toeplitz system.
\end{proof}

For further generalization, the transfer function is extended to the consideration of the Blaschke product $b(z)$, which corresponds to the all-pass filter in signal processing. The transfer function can be decomposed into the following form:
\begin{equation}
	h(z;\boldsymbol{\xi})=f(z;\boldsymbol{\xi})a(z;\boldsymbol{\xi})b(z) \nonumber
\end{equation}
where the Blaschke product $b(z)$ is given by 
\begin{equation}
	b(z)=\prod_{s} b(z,z_s)=\prod_{s} \frac{|z_s|}{z_s}\frac{z_s-z}{1-\bar{z}_s z}\nonumber
\end{equation}
and every $z_s$ is on the unit disk. Although the Blaschke product can be written in $z^{-1}$ instead of $z$, our conclusion is not changed, and we choose $z$ for our convention. When $z_s=0$, the Blaschke product is given by $b(z,z_s)=z$. Regardless of $z_s$, the Blaschke product is analytic on the unit disk. Since the Taylor expansion of the Blaschke product provides positive order terms in $z$, it is also possible to incorporate the Blaschke product into $a(z;\boldsymbol{\xi})$. However, the Blaschke product is separately considered in the paper.

The logarithmic transfer function of a linear system is represented in terms of $f, a$ and $b$: 
\begin{align}
	\log{h(z;\boldsymbol{\xi})}&=\log{(f_0a_0)}+\log{(1+\sum_{r=1}^{\infty} 
\frac{f_r}{f_0} z^{-r})}+\log{(1+\sum_{r=1}^{\infty} \frac{a_r}{a_0} z^{r})}+\log{b(z)} \nonumber \\
	&=\phi_0+\sum_{s}\log{|z_s|}+\sum_{r=1}^{\infty} \phi_r(\boldsymbol{\xi}) z^{-r}+\sum_{r=1}^{\infty} \alpha_r(\boldsymbol{\xi}) z^{r}+\sum_{r=1}^{\infty}\beta_r z^r \nonumber
\end{align}
where $\phi_0=\log{(f_0a_0)}$ and $\phi_r, \alpha_r$ are the $r$-th coefficients of the logarithmic expansions. $\phi_r$ and $\alpha_r$ are functions of $\boldsymbol{\xi}$ unless all $f_r/f_0$ and $a_r/a_0$ are constant. Meanwhile, $\beta_r=\frac{1}{r}\sum_{s}\frac{|z_s|^{2r}-1}{z_s^{r}}$ is a constant in $\boldsymbol{\xi}$.

It is also straightforward to show that the information geometry is independent of the Blaschke product.
\begin{lem}
	\label{lem_blaschke} 
	The information geometry of a signal filter is independent of the Blaschke product.
\end{lem}

\begin{proof}
It is obvious that the Blaschke product is independent of the coordinate system $\boldsymbol{\xi}$. Plugging the above series into the expression of the metric tensor in complex coordinates, Equations (\ref{metric_complex_coordinate1}) and (\ref{metric_complex_coordinate2}), the metric tensor components are expressed in terms of $\phi_r$ and $\alpha_r$: 
\begin{align}
	g_{ij}&=\partial_i \phi_0\partial_j \phi_0+\sum_{r=1}^{\infty} \partial_i \phi_r \partial_{j}\alpha_r+\sum_{r=1}^{\infty} \partial_i \alpha_r \partial_{j} \phi_r \nonumber\\
	g_{i\bar{j}}&=\sum_{r=0}^{\infty} \partial_i \phi_r \partial_{\bar{j}} \bar{\phi}_r+\sum_{r=1}^{\infty} \partial_i \alpha_r \partial_{\bar{j}} \bar{\alpha}_r \nonumber
\end{align}
and it is noteworthy that the metric tensor components are independent of the $\beta_r$ terms, which are related to the Blaschke product, because those are not functions of $\boldsymbol{\xi}$. This is why the $z$-convention for the Blaschke product is not important. It is straightforward to repeat the same calculation for the $\alpha$-connection. Based on these, the information geometry of a linear system is independent of the Blaschke product.
\end{proof}

According to Lemma \ref{lem_gauge}, the geometry is invariant under the degree of freedom in choosing $a(z;\boldsymbol{\xi})$. By using the invariance of the geometry, it is possible to fix the degree of freedom as $a_r/a_0$ constant. With the choice of the degree of freedom, the metric tensor components of the information manifold are given by
\begin{align} 
\label{metric_tensor_gauge1}
	g_{ij}&=\partial_i \phi_0 \partial_j \phi_0 \\
\label{metric_tensor_gauge2}
	g_{i\bar{j}}&=\sum_{r=0}^{\infty} \partial_i \phi_r \partial_{\bar{j}} \bar{\phi}_r
\end{align}
and it is easy to verify that the metric tensor components are only dependent on $\phi_r$ and $\bar{\phi}_r$. In other words, the metric tensor is dependent only on the unilateral part of the transfer function and a constant term in $z$ of the analytic part.

By Lemma \ref{lem_transfer_factorization}, any transfer function with the upper-bounded degree in $z$ is reducible to a unilateral transfer function with a constant term. For this class of transfer functions, a similar expression for the metric tensor can be obtained. First of all, the logarithmic transfer function is given in the series expansion:
\begin{align}
	\log{h(z;\boldsymbol{\xi})}&=\log{z^R}+\log{h_{-R}}+\log{(1+\sum_{r=1}^{\infty}\frac{h_{-R+r}}{h_{-R}}z^{-r})} \nonumber \\
	&=\log{z^R}+\sum_{r=0}^{\infty}\eta_r z^{-r} \nonumber
\end{align}
where $R$ is the highest degree in $z$. The coefficients $\eta_r$ are also known as the complex cepstrum \cite{Oppenheim:1965}, and $\eta_0=\log{h_{-R}}$. After the series expansion of this logarithmic transfer function is plugged into the formulae of the metric tensor components, Equations (\ref{metric_complex_coordinate1}) and (\ref{metric_complex_coordinate2}), the metric tensor components are obtained as
\begin{align} 
	\label{metric_tensor_highest1}
	g_{ij}&=\partial_i \eta_0\partial_j \eta_0 \\
	\label{metric_tensor_highest2}
	g_{i\bar{j}}&=\sum_{r=0}^{\infty} \partial_i \eta_r \partial_{\bar{j}} \bar{\eta}_r
\end{align}
and these expressions for the metric tensor components are similar to Equations (\ref{metric_tensor_gauge1}) and (\ref{metric_tensor_gauge2}) with the exchange of $\phi_r\leftrightarrow\eta_r$.

As an extension of Lemma \ref{lem_blaschke}, it is possible to generalize it to the inner-outer factorization of the $H^2$-functions. A function in the $H^{2}$-space can be expressed as the product of outer and inner functions by the Beurling factorization \cite{Beurling:1949}. The generalization with the Beurling factorization is given by the following lemma.
\begin{lem}
\label{lem_io_factor} 
	The information geometry of a signal filter is independent of the inner function.
\end{lem}

\begin{proof}
A transfer function $h(z;\boldsymbol{\xi })$ in the $H^{2}$-space can be decomposed by the inner-outer factorization:
\begin{equation}
	h(z;\boldsymbol{\xi })=\mathcal{O}(z;\boldsymbol{\xi })\mathcal{I}(z;\boldsymbol{\xi })\nonumber
\end{equation}
where $\mathcal{O}(z;\boldsymbol{\xi })$ is an outer function and $\mathcal{I}(z;\boldsymbol{\xi })$ is an inner function. The $\alpha $-divergence is expressed with $S(z;\boldsymbol{\xi })=|h(z;\boldsymbol{\xi })|^{2}=|\mathcal{O}(z;\boldsymbol{\xi })\mathcal{I}(z;\boldsymbol{\xi })|^{2}=|\mathcal{O}(z;\boldsymbol{\xi })|^{2}$ on the unit circle, because the inner function has the unit modulus on the unit circle. Since the $\alpha $-divergence is represented only with the outer function, other geometric objects, such as the metric tensor and the $\alpha $-connection, are also independent of the inner function.
\end{proof}

\section{K\"ahler Manifold for Signal Processing}
\label{sec_kahler_theory} 

	An advantage of the transfer function representation in the complex $z$-domain is that it is easy to test whether or not the information geometry of a given signal processing filter is a K\"ahler manifold. As mentioned before, choosing the coefficients in $a(z;\boldsymbol{\xi}) $ is considered as fixing the degrees of freedom in calculation without changing any geometry. By setting $a(z;\boldsymbol{\xi})/a_0(\boldsymbol{\xi})$ a constant function in $\boldsymbol{\xi}$, the description of a statistical model becomes much simpler, and the emergence of K\"ahler manifolds can be easily verified. Since causal filters are our main concerns in practice, we concentrate on unilateral transfer functions. Although we will work with causal filters, the results in this section are also valid for the cases of bilateral transfer functions.

\begin{thm}
\label{thm_kahler_signal}
	For a signal filter with a finite complex cepstrum norm, the information geometry of the signal filter is a K\"ahler manifold.
\end{thm}
\begin{proof}
	The information manifold of a signal filter is described by the metric tensor $g$ with the components of the expressions, Equation (\ref{metric_tensor_gauge1}) and Equation (\ref{metric_tensor_gauge2}). Any complex manifold admits a Hermitian manifold by introducing a new metric tensor $\hat{g}$ \cite{Nakahara:2003}:
		\begin{equation}
		\hat{g}_p(X,Y)=\frac{1}{2}\big(g_p(X,Y)+g_p(J_pX,J_pY)\big)\nonumber
	\end{equation}
	where $X, Y$ are tangent vectors at point $p$ on the manifold and $J$ is the almost complex structure, such that

		\begin{equation}
		J_p \frac{\partial}{\partial \xi^i}=i\frac{\partial}{\partial \xi^i} \textrm{, } J_p\frac{\partial}{\partial \bar{\xi}^i}=-i\frac{\partial}{\partial \bar{\xi}^i}. \nonumber
	\end{equation}
	With the new metric tensor $\hat{g}$, it is straightforward to verify that the information manifold is equipped with the Hermitian structure:
		\begin{align}
		\hat{g}_{ij}&=\hat{g}(\partial_i,\partial_j)=0\nonumber\\
		\hat{g}_{i\bar{j}}&=\hat{g}(\partial_i,\partial_{\bar{j}})=g_{i\bar{j}}.\nonumber
	\end{align}
	Based on the above metric tensor expressions, it is obvious that the information geometry of a linear system is a Hermitian manifold.
	
	The K\"ahler two-form $\Omega$ of the manifold is given by
	\begin{equation}
		\Omega =i\hat{g}_{i\bar{j}}d\xi ^{i}\wedge d\bar{\xi}^{j}\nonumber
	\end{equation}
	where $\wedge$ is the wedge product. By plugging Equation (\ref{metric_tensor_gauge2}) into $\Omega$, it is easy to check that the K\"ahler two-form is closed by satisfying $\partial_{k} \hat{g}_{i\bar{j}}=\partial _{i} \hat{g}_{k\bar{j}}$ and $\partial _{\bar{k}} \hat{g}_{i\bar{j}}=\partial _{\bar{j}} \hat{g}_{i\bar{k}}$.
	
	Since K\"ahler manifolds are defined as the Hermitian manifolds with the closed K\"ahler two-forms, the information geometry of a signal filter is a K\"ahler manifold.
\end{proof}

	An information manifold for a linear system with purely real parameters is a submanifold of a K\"ahlerian information manifold where the metric tensor has the isometry of exchanging holomorphic- and anti-holomorphic coordinates. In addition to that, a given linear system can be described by two manifolds: one is K\"ahler, and another is non-K\"ahler. Although the dimension is doubled, working with K\"ahler manifolds has many advantages, which will be reiterated later.

	In Theorem \ref{thm_kahler_signal}, the Hermitian condition is clearly seen after introducing the new metric tensor $\hat{g}$. It is also possible to find a condition for which the metric tensor $g$ shows the explicit Hermitian structure. To impose the explicit Hermitian condition, the following theorem is worthwhile to mention.
\begin{thm}
\label{thm_tf_fa} 
	In the K\"ahlerian information geometry of a signal filter, the Hermitian structure is explicit in the metric tensor if and only if $\phi_0$ (or $f_0a_0$) is a constant in $\boldsymbol{\xi}$. Similarly, for the transfer function of which the highest degree in $z$ is finite, the Hermitian condition is directly found if and only if the coefficient of the highest degree in $z$ of the logarithmic transfer function is a constant in $\boldsymbol{\xi}$.
\end{thm}

\begin{proof}
Let us prove the first statement. 

$(\Rightarrow)$ If the geometry is K\"ahler, it should be the Hermitian manifold satisfying 
\begin{equation}
	g_{ij}=\partial_i \phi_0\partial_j \phi_0=0 \nonumber
\end{equation}
for all $i$ and $j$. This equation exhibits that $f_0a_0$ is a constant in $\boldsymbol{\xi}$, because $\phi_0=\log{(f_0a_0)}$.

$(\Leftarrow )$ If $\phi_0$ (or $f_{0}a_{0}$) is a constant in $\boldsymbol{\xi }$, the metric tensor is found from Equations (\ref{metric_tensor_gauge1}) and  (\ref{metric_tensor_gauge2}), 
\begin{align}
	g_{ij} &=0 \nonumber \\
	\label{kahler_metric_tensor_nonvanishing}
	g_{i\bar{j}} &=\sum_{r=0}^{\infty }\partial _{i}\phi _{r}\partial _{\bar{j}}\bar{\phi}_{r}
\end{align}
and these metric tensor conditions impose that the geometry is the Hermitian manifold. It is noteworthy that the non-vanishing metric tensor components are expressed only with $\phi _{r}$ and $\bar{\phi}_{r}$, which are functions of the impulse response functions $f_r$ in $f(z;\boldsymbol{\xi })$, the unilateral part of the transfer function. For the manifold to be a K\"ahler manifold, the K\"ahler two-form $\Omega $ needs to be a closed two-form. The condition for the closed K\"ahler two-form $\Omega$ is that $\partial_{k}g_{i\bar{j}}=\partial _{i}g_{k\bar{j}}$ and $\partial _{\bar{k}}g_{i\bar{j}}=\partial _{\bar{j}}g_{i\bar{k}}$. It is easy to verify that the metric tensor components, Equation (\ref{kahler_metric_tensor_nonvanishing}), satisfy the conditions for the closed K\"ahler two-form. The Hermitian manifold with the closed K\"ahler two-form is a K\"ahler manifold.

The proof for the second statement is straightforward, because it is similar to the proof of the first one by exchanging $\phi_r\leftrightarrow\eta_r$. Let us assume that the highest degree in $z$ is $R$. According to Lemma \ref{lem_transfer_factorization}, it is possible to reduce a bilateral transfer function with finite terms along the non-causal direction to the unilateral transfer function by using the factorization of $z^R$. After that, we need to replace $\eta_0$ with $\phi_0$ in the proof. The two theorems are equivalent.
\end{proof}

	Theorem \ref{thm_tf_fa} can be applied to submanifolds of the information manifolds. For example, a submanifold of a linear system is a K\"ahler manifold if and only if $\phi_0$ (or $f_0a_0$) is constant on the submanifold, {\em i.e.}, $\phi_0$ is a function of the coordinates orthogonal to the submanifold.

On a K\"ahler manifold, the metric tensor is derived from the following equation:
\begin{equation} 
	\label{metric_potential}
	g_{i\bar{j}}=\partial_i\partial_{\bar{j}}\mathcal{K}
\end{equation}
where $\mathcal{K}$ is the K\"ahler potential. There exists the degree of freedom in K\"ahler potential up to the holomorphic and anti-holomorphic function: $\mathcal{K}(\xi,\bar{\xi})=\mathcal{K}'(\zeta,\bar{\zeta})+\phi(\zeta)+\psi(\bar{\zeta})$. However, geometry is derived from the same relation: $g_{i\bar{j}}=\partial_i\partial_{\bar{j}}\mathcal{K}$. By using Equation (\ref{metric_potential}), the information on the geometry can be extracted from the K\"ahler potential. It is necessary to find the K\"ahler potential for the signal processing geometry. The following corollary shows how to get the K\"ahler potential for the K\"ahlerian information manifold.

\begin{crl}
	For a given K\"ahlerian information geometry, the K\"ahler potential of the geometry is the square of the Hardy norm of the logarithmic transfer function. In other words, the K\"ahler potential is the square of the complex cepstrum norm of a signal filer.
\end{crl}

\begin{proof}
Given a transfer function $h(z;\boldsymbol{\xi})$, the non-trivial components of the metric tensor for a signal processing model are given by Equation (\ref{metric_complex_coordinate2}). By using integration by parts, the metric tensor component is represented by 
\begin{equation}
	g_{i\bar{j}}=\frac{1}{2\pi i}\oint_{|z|=1}\bigg\{\partial_i\Big( \log{h(z;\boldsymbol{\xi})}\partial_{\bar{j}} \log{\bar{h}(\bar{\xi};\bar{\boldsymbol{\xi}})}\Big)- \log{h(z;\boldsymbol{\xi})}\partial_i\partial_{\bar{j}} \log{\bar{h}(\bar{\xi};\bar{\boldsymbol{\xi}})}\bigg\}\frac{dz}{z} \nonumber
\end{equation}
where the latter term goes to zero by holomorphicity. When we integrate by parts with respect to the anti-holomorphic derivative once again, the metric tensor is expressed with 
\begin{equation}
	g_{i\bar{j}}=\frac{1}{2\pi i}\oint_{|z|=1}\bigg\{\partial_i\partial_{\bar{j}}\Big( \log{h(z;\boldsymbol{\xi})} \log{\bar{h}(\bar{\xi};\bar{\boldsymbol{\xi}})}\Big)-\partial_i\Big( \partial_{\bar{j}}\log{h(z;\boldsymbol{\xi})}\log{\bar{h}(\bar{\xi};\bar{\boldsymbol{\xi}})}\Big)\bigg\}\frac{dz}{z}
\nonumber
\end{equation}
and the latter term is also zero, because $h(z;\boldsymbol{\xi})$ is a holomorphic function.

Finally, the metric tensor is obtained as 
\begin{equation}
	g_{i\bar{j}}=\partial_i\partial_{\bar{j}}\Bigg(\frac{1}{2\pi i}\oint_{|z|=1}\big( \log{h(z;\boldsymbol{\xi})}\big)\big( \log{h(z;\boldsymbol{\xi})}\big)^{*}\frac{dz}{z}\Bigg) \nonumber
\end{equation}
and by the definition of the K\"ahler potential, Equation (\ref{metric_potential}), the K\"ahler potential of the linear system geometry is given by 
\begin{equation}
	\mathcal{K}=\frac{1}{2\pi i}\oint_{|z|=1}\big( \log{h(z;\boldsymbol{\xi})}\big)\big( \log{h(z;\boldsymbol{\xi})}\big)^{*}\frac{dz}{z} \nonumber
\end{equation}
up to a holomorphic function and an anti-holomorphic function. The right-handed side of the above equation is known as the square of the Hardy norm for the logarithmic transfer function. It is straightforward to derive the relation between the K\"ahler potential and the square of the Hardy norm of the logarithmic transfer function: 
\begin{equation} 
	\label{kahler_potential_Hardy}
	\mathcal{K}=\frac{1}{2\pi i}\oint_{|z|=1}\big( \log{h(z;\boldsymbol{\xi})}\big)\big( \log{h(z;\boldsymbol{\xi})}\big)^{*}\frac{dz}{z}=\|\log h(z;\boldsymbol{\xi})\|^2_{H^2}.
\end{equation}
	Additionally, the Hardy norm of the logarithmic transfer function is also known as the complex cepstrum norm of a linear system \cite{Oppenheim:1965, Martin:2000}.
\end{proof}

For a given linear system, the K\"ahler potential of the geometry is given by $\phi_r$, $\alpha_r$ and the complex conjugates of $\phi_r$, $\alpha_r$:
\begin{equation}
	\mathcal{K}=\sum_{r=0}^{\infty} (\phi_r\bar{\phi}_r+\alpha_r\bar{\alpha}_r).\nonumber
\end{equation}
However, the geometry is not dependent on $\alpha$ and $\bar{\alpha}$, because those are not the functions of the model parameters $\boldsymbol{\xi}$ under fixing the degree of the freedom. By using Equation (\ref{kahler_metric_tensor_nonvanishing}), the K\"ahler potential is expressed with
\begin{equation}
	\mathcal{K}=\sum_{r=0}^{\infty} \phi_r\bar{\phi}_r \nonumber
\end{equation}
and it is noticeable that the K\"ahler potential only depends on $\phi_r$ and $\bar{\phi}_r$, which come from the unilateral part of the transfer function decomposition. It is possible to obtain a similar expression for the finite highest upper-degree case by changing $\phi_r$ to $\eta_r$.

Since we assume that the complex cepstrum norm is finite, a transfer function $h(z;\boldsymbol{\xi})$ in the $H^2$-space also lives in the Hardy space of 
\begin{equation}
	\mathcal{K}=\|\log h(z;\boldsymbol{\xi})\|^2_{H^2}<\infty.\nonumber
\end{equation}
This implies that the transfer function lives not only in $H^2$, but also in $\exp{(H^2)}$, equivalently $\log{h}$ in the $H^2$-space.

From Equation (\ref{metric_potential}), the metric tensor is derived from the K\"ahler potential. Additionally, the metric tensor is also calculated from the $\alpha$-divergence. These facts indicate that there exists a connection between the K\"ahler potential and the $\alpha$-divergence.

\begin{crl}
	The K\"ahler potential is a constant term in $\alpha$, up to purely holomorphic or purely anti-holomorphic functions, of the $\alpha$-divergence between a signal processing filter and the all-pass filter of a unit transfer function.
\end{crl}

\begin{proof}
After replacing the spectral density function with the transfer function, the 0-divergence between a signal filter and the all-pass filter with a unit transfer function is given by
\begin{align}
D^{(0)}(1||h) &=\frac{1}{2\pi i}\oint_{|z|=1}\frac{1}{2}(\log{h}+\log{\bar{h}})^2\frac{dz}{z} \nonumber \\
&={\mathcal{K}}+\frac{1}{2\pi i}\oint_{|z|=1}\frac{1}{2}\big((\log{h})^2+(\log{\bar{h}})^2 \big)\frac{dz}{z} \nonumber \\
&={\mathcal{K}}+F(\boldsymbol{\xi})+\bar{F}(\bar{\boldsymbol{\xi}}) \nonumber
\end{align}
where $F(\boldsymbol{\xi})=\frac{1}{2}\phi_0^2=\frac{1}{2}(\log{(f_0a_0)})^2$. For a bilateral transfer function, $F(\boldsymbol{\xi})=\frac{1}{2}(\phi_0+\sum \log{|z_s|})^2+\sum_{r=1}\phi_r(\alpha_r+\beta_r) $.

For non-zero $\alpha$, the $\alpha$-divergence between a signal and the white noise is also obtained as
\begin{align}
	D^{(\alpha)}(1||h)&=\frac{1}{2\pi i\alpha^2}\oint_{|z|=1} \big\{h^{\alpha}-1-\alpha (\log{h}+\log{\bar{h}})\big\}\frac{dz}{z} \nonumber \\
	&=\frac{1}{2\pi i}\oint_{|z|=1} \Big( \frac{1}{2}(\log{h}+\log{\bar{h}})^2+\sum_{n=1}^{\infty}\frac{1}{(n+2)!}\alpha^{n}(\log{h}+\log{\bar{h}})^{n+2}\Big)\frac{dz}{z} \nonumber \\
	&=D^{(0)}(1||h)+{\mathcal{O}}(\alpha) \nonumber \\
	&={\mathcal{K}}+F(\boldsymbol{\xi})+\bar{F}(\bar{\boldsymbol{\xi}})+{\mathcal{O}}(\alpha) . \nonumber
\end{align}
When $f_0a_0$ is unity, a constant term in $\alpha$ of the $\alpha$-divergence is the K\"ahler potential. This shows the relation between the $\alpha$-divergence and the K\"ahler potential.
\end{proof}

	The $\alpha$-connection on a K\"ahler manifold is expressed with the transfer function by using Equation (\ref{sdf_tf}) and Equation (\ref{connection_sdf}). It is also cross-checked from the $\alpha$-divergence in the transfer function representation.

\begin{crl}
The $\alpha$-connection components of the K\"ahlerian information geometry are found as
\begin{align}
	\Gamma^{(\alpha)}_{ij,\bar{k}}&=\frac{1}{2\pi i}\oint_{|z|=1}\big(\partial_i\partial_j \log{h(z;\boldsymbol{\xi})}-\alpha \partial_i \log{h(z;\boldsymbol{\xi})}\partial_j \log{h(z;\boldsymbol{\xi})}\big)\big(\partial_k\log{h(z;\boldsymbol{\xi})}\big)^{*}\frac{dz}{z} \nonumber \\
	\Gamma^{(\alpha)}_{ij,k}&=\frac{1}{2\pi i}\oint_{|z|=1}\big(\partial_i\partial_j \log{h(z;\boldsymbol{\xi})}-\alpha \partial_i \log{h(z;\boldsymbol{\xi})}\partial_j \log{h(z;\boldsymbol{\xi})}\big)\big(\partial_k\log{h(z;\boldsymbol{\xi})}\big)\frac{dz}{z} \nonumber \\
	\Gamma^{(\alpha)}_{i\bar{j},k}&=\frac{1}{2\pi i}\oint_{|z|=1}-\alpha\big(\partial_i \log{h(z;\boldsymbol{\xi})}\big)\big(\partial_{j} \log{h(z;\boldsymbol{\xi})}\big)^{*}\big(\partial_k \log{h(z;\boldsymbol{\xi})}\big)\frac{dz}{z} \nonumber \\
	\Gamma^{(\alpha)}_{i\bar{j},\bar{k}}&=\frac{1}{2\pi i}\oint_{|z|=1}-\alpha\big(\partial_i \log{h(z;\boldsymbol{\xi})}\big)\big(\partial_{\bar{j}} \log{h(z;\boldsymbol{\xi})}\big)^{*}\big(\partial_k \log{h(z;\boldsymbol{\xi})}\big)^{*}\frac{dz}{z} \nonumber
\end{align}
and the non-trivial components of the symmetric tensor $T$ are given by 
\begin{align} 
	\label{symt_transfer}
	T_{ij,\bar{k}}=\frac{1}{\pi i}\oint_{|z|=1}\big(\partial_i \log{h(z;\boldsymbol{\xi})})(\partial_j \log{h(z;\boldsymbol{\xi})}\big)\big(\partial_k\log{h(z;\boldsymbol{\xi})}\big)^{*}\frac{dz}{z}\\
	T_{ij,k}=\frac{1}{\pi i}\oint_{|z|=1}\big(\partial_i \log{h(z;\boldsymbol{\xi})})(\partial_j \log{h(z;\boldsymbol{\xi})}\big)\big(\partial_k\log{h(z;\boldsymbol{\xi})}\big)\frac{dz}{z}\nonumber.
\end{align}
In particular, the non-vanishing 0-connection components are expressed with
\begin{equation}
	\Gamma^{(0)}_{ij,\bar{k}}=(\Gamma^{(0)}_{\bar{i}\bar{j},k})^{*}=\frac{1}{2\pi i}\oint_{|z|=1}\big(\partial_i\partial_j \log{h(z;\boldsymbol{\xi})}\big)\big(\partial_k \log{h(z;\boldsymbol{\xi})}\big)^{*}\frac{dz}{z} \nonumber \\
\end{equation}
and the 0-connection is directly derived from the K\"ahler potential: 
\begin{equation} 
	\label{connection_potential}
	\Gamma^{(0)}_{ij,\bar{k}}=\partial_i\partial_j\partial_{\bar{k}}\mathcal{K}.
\end{equation}
Additionally, the $\alpha$-connection and the $(-\alpha)$-connection are dual to each other.
\end{crl}

\begin{proof}
After plugging Equation (\ref{sdf_tf}) into Equation (\ref{connection_sdf}), the
derivation of the $\alpha$-connection is straightforward by considering holomorphic and anti-holomorphic derivatives in the expression. The same procedure is applied to the derivation of the symmetric tensor $T$.

The 0-connection is also directly derived from the K\"ahler potential. The proof is as follows: 
\begin{align}
	\Gamma^{(0)}_{ij,\bar{k}}&=\frac{1}{2\pi i}\oint_{|z|=1}\big(\partial_i\partial_j \log{h(z;\boldsymbol{\xi})}\big)\big(\partial_k \log{h(z;\boldsymbol{\xi})}\big)^{*}\frac{dz}{z} \nonumber \\
	&=\partial_i\partial_j\partial_{\bar{k}}\Big(\frac{1}{2\pi i}\oint_{|z|=1}\big(\log{h(z;\boldsymbol{\xi})}\big)\big( \log{h(z;\boldsymbol{\xi})}\big)^{*}\frac{dz}{z}\Big) \nonumber \\
	&=\partial_i\partial_j\partial_{\bar{k}}\big(||\log h(z;\boldsymbol{\xi})||^2_{H^2}\big) \nonumber \\
	&=\partial_i\partial_j\partial_{\bar{k}}\mathcal{K}.\nonumber
\end{align}

To prove the $\alpha$-duality, we need to test the following relation:
\begin{equation}
	\partial_{\mu}g_{\nu\rho}=	\Gamma^{(\alpha)}_{\mu\nu,\rho}+\Gamma^{(-\alpha)}_{\mu\rho,\nu}\nonumber
\end{equation}
where the Greek letters run from $1,\cdots,n$, $\bar{1},\cdots,\bar{n}$. After tedious calculation, it is obvious that the above equation is satisfied regardless of combinations of the indices. Therefore, the $\alpha$-duality also exists on the K\"ahlerian information manifolds.
\end{proof}

The 0-connection and the symmetric tensor $T$ are expressed in terms of $\phi_r$ and $\bar{\phi}_r$,
\begin{align}
	\Gamma^{(0)}_{ij,\bar{k}}&=\sum_{r=0}^{\infty} \partial_i\partial_j\phi_r\partial_{\bar{k}}\bar{\phi}_r \nonumber \\
	\Gamma^{(0)}_{ij,k}&=\partial_i\partial_j\phi_0\partial_{k}\phi_0 \nonumber \\
	T_{ij,\bar{k}}&=2 \sum_{r,s=0}^{\infty}\partial_i \phi_r \partial_j\phi_s\partial_{\bar{k}}\bar{\phi}_{r+s} \nonumber\\
	T_{ij,k}&=2 \partial_i \phi_0 \partial_j\phi_0\partial_{k}\phi_{0}.\nonumber
\end{align}
With the degree of freedom that $\phi_0$ is a constant in the model parameters $\boldsymbol{\xi}$, the non-trivial components of the 0-connection and the symmetric tensor $T$ are $\Gamma^{(0)}_{ij,\bar{k}}$ and $T_{ij,\bar{k}}$, respectively. In this degree of freedom, the Hermitian condition on the metric tensor is obviously emergent, and it is also beneficial to check the $\alpha$-duality condition for non-vanishing components:
\begin{align}
	\partial_k g_{i\bar{j}}&=\Gamma^{(\alpha)}_{ki,\bar{j}}+\Gamma^{(-\alpha)}_{k\bar{j},i}\nonumber\\
	\partial_{\bar{k}} g_{i\bar{j}}&=\Gamma^{(\alpha)}_{\bar{k}i,\bar{j}}+\Gamma^{(-\alpha)}_{\bar{k}\bar{j},i}.\nonumber
\end{align}

	We can cross-check these formulae for the geometric objects of the linear system geometry with the well-known results on a K\"ahler manifold. First of all, the fact that the 0-connection is the Levi--Civita connection can be verified as follows: 
\begin{equation}
	\Gamma^{(0)}{}_{ij}^{k}=g^{k\bar{m}}\Gamma^{(0)}_{ij,\bar{m}}=g^{k\bar{m}}\partial_i\partial_j\partial_{\bar{m}}\mathcal{K}=g^{k\bar{m}}\partial_i g_{j\bar{m}}=\partial_i (\log{g_{m\bar{n}}})^{k}_{\phantom{1}j}=\Gamma_{ij}^{k} \nonumber
\end{equation}
where the last equality comes from the expression for the Levi--Civita connection on a K\"ahler manifold. This is well-matched to the Levi--Civita connection on a K\"ahler manifold.

In Riemannian geometry, the Riemann curvature tensor, corresponding to the 0-curvature tensor, is given by 
\begin{equation}
R^{\rho}_{\phantom{1}\sigma\mu\nu}=\partial_{\mu}\Gamma^{\rho}_{\nu\sigma}-\partial_{\nu}\Gamma^{\rho}_{\mu\sigma}+\Gamma^{\rho}_{\mu\lambda}\Gamma^{\lambda}_{\nu\sigma}-\Gamma^{\rho}_{\nu\lambda}\Gamma^{\lambda}_{\mu\sigma} 
\nonumber
\end{equation}
where the Greek letters can be any holomorphic and anti-holomorphic indices. Similar to a Hermitian manifold, the non-vanishing components of the 0-curvature tensor on a K\"ahler manifold are $R^{\rho}_{\sigma\bar{\mu}j}$ and its complex conjugate, {\em i.e.}, the components with three holomorphic indices and one anti-holomorphic index (and the complex conjugate component). The non-trivial components of the Riemann curvature tensor are represented by 
\begin{align}
	R^{(0)}{}^{l}_{\phantom{1}k\bar{i}j}&=\partial_{\bar{i}}\Gamma^{l}_{jk}-\partial_{j}\Gamma^{l}_{\bar{i}k}+\Gamma^{l}_{\bar{i}m}\Gamma^{m}_{jk}-\Gamma^{l}_{jm}\Gamma^{m}_{\bar{i}k} \nonumber \\
	&=\partial_{\bar{i}}\Gamma^{l}_{jk}=\partial_{\bar{i}}(g^{l\bar{m}}\partial_j\partial_l\partial_{\bar{m}}\mathcal{K})=\big(R^{(0)}{}^{\bar{l}}_{\phantom{1}\bar{k}i\bar{j}}\big)^{*} \nonumber
\end{align}
because the 0-connection components with the mixed indices are vanishing.

Taking index contraction on holomorphic upper and lower indices in the Riemann curvature tensor, the 0-Ricci tensor is found as
\begin{align} 
	\label{Ricci_tensor_kahler}
	R^{(0)}_{i\bar{j}}&=R^{(0)}{}^{k}_{\phantom{1}k i\bar{j}}=-R^{(0)}{}^{k}_{\phantom{1}k\bar{j}i} \nonumber \\
	&=-\partial_{\bar{j}}\partial_{i} (\log{g_{m\bar{n}}})^{k}_{\phantom{1}k}=-\partial_{\bar{j}}\partial_{i}tr(\log{g_{m\bar{n}}}) \nonumber \\
	&=-\partial_{\bar{j}}\partial_{i} \log{\mathcal{G}}
\end{align}
where $\mathcal{G}$ is the determinant of the metric tensor. This result is consistent with the expression of the Ricci tensor on a K\"ahler manifold. It is also straightforward to obtain the 0-scalar curvature by contracting the indices in the 0-Ricci tensor:
\begin{equation}
	R^{(0)}=g^{i\bar{j}}R^{(0)}_{i\bar{j}}=-\frac{1}{2}\Delta \log{\mathcal{G}}\nonumber
\end{equation}
where $\Delta$ is the Laplace--Beltrami operator on the K\"ahler manifold.

	The $\alpha$-generalization of the curvature tensor, the Ricci tensor and the scalar curvature is based on the $\alpha$-connection, Equation (\ref{alpha_connection_t}). The $\alpha$-curvature tensor is given by 
\begin{align}
	R^{(\alpha)}{}^{l}_{\phantom{1}k\bar{i}j}&=\partial_{\bar{i}}\Gamma^{(\alpha)}{}^{l}_{jk}=\partial_{\bar{i}}\Big(\Gamma^{(0)}{}^{l}_{jk}-\frac{\alpha}{2} g^{l\bar{m}}T_{jk,\bar{m}}\Big) \nonumber \\
	&=R^{(0)}{}^{l}_{\phantom{1}k\bar{i}j}-\frac{\alpha}{2} \partial_{\bar{i}}\Big(g^{l\bar{m}}T_{jk,\bar{m}}\Big). \nonumber
\end{align}
The $\alpha$-Ricci tensor and the $\alpha$-scalar curvature are obtained as 
\begin{align}
	R^{(\alpha)}_{i\bar{j}}&=R^{(\alpha)}{}^{k}_{\phantom{1}k i\bar{j}}=-R^{(\alpha)}{}^{k}_{\phantom{1}k\bar{j}i} \nonumber \\
	&=-\partial_{\bar{j}}\Big(\Gamma^{(0)}{}^{k}_{ik}-\frac{\alpha}{2} g^{k\bar{l}}T_{ik,\bar{l}}\Big) \nonumber \\
	&=R^{(0)}_{i\bar{j}}+\frac{\alpha}{2}\partial_{\bar{j}}T^{k}_{ik} \nonumber \\
	R^{(\alpha)}&=R^{(0)}+\frac{\alpha}{2} g^{i\bar{j}}\partial_{\bar{j}}T^{\rho}_{i\rho}. \nonumber
\end{align}
It is noteworthy that the $\alpha$-curvature tensor, the $\alpha$-Ricci tensor and the $\alpha$-scalar curvature on a K\"ahler manifold have the linear corrections in $\alpha$ comparing with the quadratic corrections in $\alpha$ on non-K\"ahler manifolds.
	A submanifold of a K\"ahler manifold is also a K\"ahler manifold. When a submanifold of dimension $m$ exists, the transfer function of a linear system can be decomposed into two parts: 
\begin{equation}
	h(z;\boldsymbol{\xi})=h_{\parallel}(z;\xi^1,\cdots,\xi^{m})h_{\perp}(z;\xi^{m+1},\cdots,\xi^{n}) \nonumber
\end{equation}
where $h_{\parallel}$ is the transfer function on the submanifold and $h_{\perp}$ is the transfer function orthogonal to the submanifold. When it is plugged into Equation (\ref{kahler_potential_Hardy}), the K\"ahler potential of the geometry is decomposed into three terms as follows:
\begin{align}
	\mathcal{K}&=\frac{1}{2\pi i}\oint_{|z|=1}(\log{h_{\parallel}}+\log{h_{\perp}})(\log{h_{\parallel}}+\log{h_{\perp}})^{*}\frac{dz}{z} \nonumber \\
	&=\frac{1}{2\pi i}\oint_{|z|=1}\log{h_{\parallel}}\log{\bar{h}_{\parallel}}\frac{dz}{z}+\frac{1}{2\pi i}\oint_{|z|=1}\log{h_{\perp}}\log{\bar{h}_{\perp}}\frac{dz}{z} +\frac{1}{2\pi i}\oint_{|z|=1}\log{h_{\parallel}}\log{\bar{h}_{\perp}}\frac{dz}{z}+(c.c.) \nonumber \\
	&=\mathcal{K}_{\parallel}+\mathcal{K}_{\perp}+\mathcal{K}_{\times} \nonumber
\end{align}
where $\mathcal{K}_{\parallel}$ contains the coordinates from the submanifold, $\mathcal{K}_{\times}$ is for the cross-terms and $\mathcal{K}_{\perp}$ is orthogonal to the submanifold.

It is obvious that each part in the decomposition of the K\"ahler potential provides the metric tensors for submanifolds,
\begin{align}
g_{M\bar{N}}&=\partial_{M}\partial_{\bar{N}}\mathcal{K}_{\parallel} \nonumber\\
g_{M\bar{n}}&=\partial_{M}\partial_{\bar{n}}\mathcal{K}_{\times} \nonumber \\
g_{m\bar{n}}&=\partial_{m}\partial_{\bar{n}}\mathcal{K}_{\perp} \nonumber
\end{align}
where an uppercase index is for the coordinates on the submanifold and a lowercase index is for the coordinates orthogonal to the submanifold. As we already know, the induced metric tensor for the submanifold is derived from $\mathcal{K}_{\parallel}$, the K\"ahler potential of the submanifold. Based on this decomposition, it is also possible to use $\mathcal{K}$ as the K\"ahler potential of the submanifold, because it endows the same metric with $\mathcal{K}_{\parallel}$. However, the Riemann curvature tensor and the Ricci tensors include the mixing terms from embedding in the ambient manifold, because the inverse metric tensor contains the orthogonal coordinates by the Schur complement. In statistical inference, connections, tensors and scalar curvature play important roles. If those corrections are negligible, dimensional reduction to the submanifolds is meaningful from the viewpoints not only of K\"ahler geometry, but also of statistical inference.

The benefits of introducing a K\"ahler manifold as an information manifold are as follows. First of all, on a K\"aher manifold, the calculation of geometric objects, such as the metric tensor, the $\alpha$-connection and the Ricci tensor, is simplified by using the K\"ahler potential. For example, the 0-connection on a non-K\"ahler
manifold is given by 
\begin{equation}
	\Gamma^{(0)}_{ij,k}=\frac{1}{2}(\partial_i g_{kj}+\partial_j g_{ik}-\partial_k g_{ij}) \nonumber
\end{equation}
demanding three-times more calculation steps than the K\"ahler case, Equation (\ref{connection_potential}). Additionally, the Ricci tensor on a K\"ahler manifold is directly derived from the determinant of the metric tensor. Meanwhile, the Ricci tensor on a non-K\"ahler manifold needs more procedures. In the beginning, the connection should be calculated from the metric tensor. Additionally, then, the Riemann curvature is obtained after taking the derivatives on the connection and considering quadratic terms of the connection. Finally, the Ricci tensor on the non-K\"ahler manifold is found by the index contraction on the curvature tensor indices.

Secondly, $\alpha$-corrections on the Riemann curvature tensor, the Ricci tensor and the scalar curvature on the K\"ahler manifold are linear in $\alpha$. Meanwhile, there exist the quadratic $\alpha$-corrections in non-K\"ahler cases. The $\alpha$-linearity makes it much easier to understand the properties of $\alpha$-family.

Moreover, submanifolds in K\"ahler geometry are also K\"ahler manifolds. When a statistical model is reducible to its lower-dimensional models, the information geometry of the reduced statistical model is a submanifold of the geometry. If the ambient manifold is K\"ahler, the dimensional reduction also provides a K\"ahler manifold as the information geometry of the reduced model, and the submanifold is equipped with all of the properties of the K\"ahler manifold.

Lastly, finding the superharmonic priors suggested by Komaki \cite{Komaki:2006} is more straightforward in the K\"ahler setup, because the Laplace--Beltrami operator on a K\"ahler manifold is of the more simplified form compared to that in non-K\"ahler cases. For a differentiable function $\psi$, the Laplace--Beltrami operator on a K\"ahler manifold is given by
\begin{equation}
	\Delta \psi =2g^{i\bar{j}}\partial_{i}\partial_{\bar{j}} \psi
\end{equation}
comparing with the Laplace--Beltrami operator on a non-K\"ahler manifold: 
\begin{equation}
	\Delta \psi =\frac{1}{\sqrt{\mathcal{G}}}\partial _{i}\big(\sqrt{\mathcal{G}}g^{ij}\partial_{j}\psi \big)
\end{equation}
where $\mathcal{G}$ is the determinant of the metric tensor. On a K\"ahler manifold, the partial derivatives only act on the superharmonic prior functions. Meanwhile, the contributions from the derivatives acting on $\mathcal{G}$ and $g^{ij}$ should be considered in the non-K\"ahler cases. This computational redundancy is not on the K\"ahler manifold.

\section{Example: AR, MA and ARMA Models}
\label{sec_kahler_example}
 
	In the previous section, we show that the information geometry of a signal filter is a K\"ahler manifold. From the viewpoint of signal processing, time series models can be interpreted as a signal filter that transforms a randomized input $x(z)$ to an output $y(z)$. The geometry of a time series model can also be found by using the results in the previous section. In particular, we cover the AR, the MA and the ARMA models as examples.

First of all, the transfer functions of these time series models need to be identified. The transfer functions of the AR, the MA and the ARMA models with model parameters $\boldsymbol{\xi}=(\sigma, \xi^1,\cdots, \xi^n)$ are represented by 
\begin{equation}
	h(z;\boldsymbol{\xi})=\frac{\sigma^2}{2\pi}\prod_{i=1}^{n} (1-\xi^i z^{-1})^{c_i} 
\nonumber
\end{equation}
where $c_i=-1$ if $\xi^i$ is an AR pole and $c_i=1$ if $\xi^i$ is an MA root.

The ARMA models can be considered as the fraction of two AR models or two MA models. By Lemma \ref{lem_sdf_inv}, the correspondence between the $\alpha$-duality and the reciprocality of transfer functions is also valid for the ARMA($p,q$) models. For example, the ARMA($p,q$) model with $\alpha$-connection is $\alpha$-dual to the ARMA($q,p$) model with the $(-\alpha)$-connection under the reciprocality of the transfer function. Simply speaking, the AR model and the MA model are exchangeable by Lemma \ref{lem_sdf_inv}. The correspondence is given as follows: 
\begin{align}
\text{ARMA}(p,q) &\leftrightarrow \text{ARMA}(q,p) \nonumber \\
\text{poles} &\leftrightarrow \text{zeros} \nonumber \\
\text{zeros} &\leftrightarrow \text{poles} \nonumber \\
\sigma/\sqrt{2\pi} &\leftrightarrow \sqrt{2\pi}/\sigma \nonumber \\
\alpha &\leftrightarrow -\alpha \nonumber \\
\Gamma^{(\alpha)} &\leftrightarrow \Gamma^{(-\alpha)} \nonumber \\
D^{(\alpha)}(h^{(0)}||h)&\leftrightarrow D^{(-\alpha)}(h^{(0)}||h) \nonumber
\end{align}
where $h^{(0)}$ is the unit transfer function of an all-pass filter.

\subsection{K\"ahlerian Information Geometry of ARMA($p,q$) Models}

The ARMA($p,q$) model is the ($p$+$q$+1)-dimensional model with $\boldsymbol{\xi}=(\sigma, \xi^1,\cdots,\xi^{p+q})$, and the time series model is characterized by its transfer function: 
\begin{equation}
	h(z;\boldsymbol{\xi})=\frac{\sigma^2}{2\pi}\frac{(1-\xi^{p+1} z^{-1})(1-\xi^{p+2} z^{-1})\cdots(1-\xi^{p+q} z^{-1})}{(1-\xi^1 z^{-1})(1-\xi^2 z^{-1})\cdots(1-\xi^p z^{-1})} \nonumber
\end{equation}
where $\sigma$ is the gain and $\xi^i$ is a pole with the condition of $|\xi^i|<1$. The logarithmic transfer function of the ARMA($p,q$) model is given by
\begin{equation}
	\log{h(z;\boldsymbol{\xi})}=\log{\frac{\sigma^2}{2\pi}}+\sum_{i=1}^{p+q} c_i\log{(1-\xi^i z^{-1})} \nonumber
\end{equation}
and it is easy to verify that $f_0 a_0=\sigma^2/2\pi$. 

According to Theorem \ref{thm_kahler_signal}, the information geometry of the ARMA model is a K\"ahler manifold because of stability, minimum phase and the finite complex cepstrum norm of the ARMA filter. By using Theorem \ref{thm_tf_fa}, the Hermitian condition on the metric tensor is explicitly checked on the submanifold of the ARMA model, where $\sigma$ is a constant. In addition to that, this submanifold is also a K\"ahler manifold, because a submanifold of a K\"ahler manifold is also K\"ahler. Since it is possible to gauge $\sigma$ by normalizing the amplitude of an input signal, the $\sigma$-coordinate can be considered as the denormalization coordinate \cite{Amari:2000}. Similar to the non-complexified ARMA models \cite{Ravishanker:1990p5895}, $g_{0i}$ for all non-zero $i$ vanish by direct calculation using Equation (\ref{metric_sdf}). Considering these facts, we work only with the submanifolds of a constant gain.

As mentioned, the K\"aher potential is crucial for the K\"ahler manifolds and defined as the square of the Hardy norm of the logarithmic transfer function, equivalently the square of the complex cepstrum norm, Equation (\ref{kahler_potential_Hardy}). For the ARMA($p,q$) model, the K\"ahler potential is given by 
\begin{equation}
	\mathcal{K}=\sum_{r=1}^{\infty}\frac{1}{r^2}\Big|\sum_{i=1}^{p+q}c_i(\xi^i)^r\Big|^2 \nonumber
\end{equation}
Since the metric tensor is simply derived from taking the partial derivatives on the K\"ahler potential, Equation (\ref{metric_potential}), the metric tensor of the ARMA($p,q$) model is represented as
\begin{equation}
	g_{i\bar{j}}=\frac{c_i c_j}{1-\xi^i\bar{\xi}^j}.\nonumber
\end{equation}
where other fully holomorphic- and fully anti-holomorphic-indexed components are all zero. It is easily verified that if $c_i$ and $c_j$ are both from the AR or the MA models, $c_i$ and $c_j$ exhibit the same signature, which imposes that the AR($p$)- and the MA($q$)-submanifolds of the ARMA($p,q$) model have the same metric tensors with the AR($p$) and the MA($q$) models, respectively. If two indices are from the different models, there exists only the sign difference in the metric tensor. The metric tensor of the geometry is of a similar form as the metric tensor in Ravishanker's work on the ARMA geometry \cite{Ravishanker:1990p5895}.

By considering the Schur complement, the inverse metric tensor can be deduced from the inverse metric tensor of the AR($p$+$q$) model. The inverse metric tensor of the geometry is represented by 
\begin{equation}
	g^{i\bar{j}}=c_ic_j\frac{(1-\xi^i\bar{\xi}^j)\prod_{k\ne i}(1-\xi^k\bar{\xi}^j)\prod_{k\ne j}(1-\xi^i\bar{\xi}^k)}{\prod_{k\ne i}(\xi^k-\xi^i)\prod_{k\ne j}(\bar{\xi}^k-\bar{\xi}^j)} \nonumber
\end{equation}
and the only difference with the AR case is the signature $c_ic_j$ in the AR-MA mixed components. With the sign difference in the metric tensor components with the AR-MA mixed indices, the determinant of the metric tensor can be calculated with the aid of the Schur complement. The determinant of the metric tensor is found as 
\begin{equation}
	\mathcal{G}=\det g_{i\bar{j}}=\frac{\prod_{1\le j<k\le n}|\xi^k-\xi^j|^2}{\prod_{j,k}(1-\xi^j\bar{\xi}^k)}.\nonumber
\end{equation}

The 0-connection and the symmetric tensor $T$ for the K\"ahler-ARMA model can be found from the results in the previous section. The non-trivial 0-connection components are calculated from Equation~(\ref{connection_potential}): 
\begin{equation}
\Gamma_{ij,\bar{k}}^{(0)}=\frac{c_jc_k\delta_{ij}\bar{\xi}^{k}}{(1-\xi^j\bar{\xi}^{k})^2} \nonumber
\end{equation}
and the non-zero components of the symmetric tensor $T$ are given by Equation (\ref{symt_transfer}): 
\begin{equation}
T_{ij,\bar{k}}=\frac{2c_ic_jc_k\bar{\xi}^{k}}{(1-\xi^i\bar{\xi}^{k})(1-\xi^j\bar{\xi}^{k})}. \nonumber
\end{equation}
Based on the above expressions, the $\alpha$-connection is easily obtained from Equation (\ref{alpha_connection_t}).

The 0-Ricci tensor of the ARMA geometry is represented by Equation (\ref{Ricci_tensor_kahler}):
\begin{equation}
	R^{(0)}_{i\bar{j}}=-\frac{1}{(1-\xi^i\bar{\xi}^j)^2} \nonumber
\end{equation}
and it is noteworthy that the Ricci tensor is not dependent on $c_i$. The 0-scalar curvature is calculated from the 0-Ricci tensor by index contraction:
\begin{equation}
R^{(0)}=-\sum_{i,j}\frac{c_ic_j\prod_{k\ne i}(1-\xi^k\bar{\xi}^j)\prod_{k\ne j}(1-\xi^i\bar{\xi}^k)}{(1-\xi^i\bar{\xi}^j)\prod_{k\ne i}(\xi^k-\xi^i)\prod_{k\ne j}(\bar{\xi}^k-\bar{\xi}^j)} \nonumber
\end{equation}
where $c_i, c_j$ are from the inverse metric tensor of the ARMA model.

It is straightforward to derive the $\alpha$-generalization of the Riemann curvature tensor, the Ricci tensor and the scalar curvature by using the results in Section \ref{sec_kahler_theory}.

\subsection{Superharmonic Priors for K\"ahler-ARMA($p,q$) Models}

As mentioned before, the Laplace--Beltrami operator on a K\"ahler manifold is of a much simpler form than that on a non-K\"ahler manifold. The simplified Laplace--Beltrami operator of the geometry makes finding superharmonic priors easier. Although it is also valid in any arbitrary dimension, let us confine ourselves to the ARMA(1,1) model as a simplification. For the ARMA($1,1$) model, the metric tensor is expressed with
\begin{equation}
g_{i\bar{j}}=\Big(
\begin{array}{cc}
\frac{1}{1-|\xi^1|^2} & -\frac{1}{1-\xi^1\bar{\xi}^2} \\ 
-\frac{1}{1-\xi^2\bar{\xi}^1} & \frac{1}{1-|\xi^2|^2}
\end{array}
\Big). \nonumber
\end{equation}
It is trivial to show that $\psi_1=(1-|\xi^1|^2)+(1-|\xi^2|^2)$ and $\psi_2=(1-|\xi^1|^2)(1-|\xi^2|^2)\nonumber$ are superharmonic prior functions. 

In order to compare with the literature on superharmonic priors for the non-K\"ahlerian AR \mbox{models~\cite{Tanaka:2006, Tanaka:2009}}, let us consider the K\"ahler-AR($p$) models. For $p=2$, the metric tensor is given by 
\begin{align}
	g_{i\bar{j}}=\Big( \begin{array}{cc} 
	\frac{1}{1-|\xi ^{1}|^{2}} & \frac{1}{1-\xi ^{1}\bar{\xi}^{2}} \\ 
	\frac{1}{1-\xi ^{2}\bar{\xi}^{1}} & \frac{1}{1-|\xi ^{2}|^{2}}
	\end{array}\Big)\nonumber.
\end{align}

With the Laplace--Beltrami operator on a K\"ahler manifold, it is obvious that $(1-|\xi^k|^2)$ for $k=1,\cdots,p$ is a superharmonic function in arbitrary $p$-dimensional AR geometry. The proof for superharmonicity is as follows:
\begin{align}
	\Delta (1-|\xi^k|^2)&=2g^{i\bar{j}}\partial_i\partial_{\bar{j}} (1-|\xi^k|^2) \nonumber \\
	&=-2g^{i\bar{j}}\delta_{i,k}\delta_{j,k}=-2g^{k\bar{k}}<0 \nonumber
\end{align}
because the diagonal components of the inverse metric tensor are all positive. By additivity, the sum of these prior functions, $\sum_{k=1}^{n}(1-|\xi^k|^2)$, are also superharmonic. Obviously, $\psi_1=(1-|\xi^1|^2)+(1-|\xi^2|^2)$ is a superharmonic prior function in the two-dimensional case.

Another superharmonic prior function for the AR($2$) model is $\psi_2=(1-|\xi^1|^2)(1-|\xi^2|^2)$. The Laplace--Beltrami operator acting on $\psi_2$ is represented by 
\begin{equation}
	\Big(\frac{\Delta \psi_2}{\psi_2}\Big)=-\frac{2(2-\xi^1\bar{\xi}^2-\xi^2\bar{\xi}^1)}{|\xi^1-\xi^2|^2} \nonumber
\end{equation}
and it is simply verified that $\Big(\frac{\Delta \psi_2}{\psi_2}\Big)<0$, because $2-\xi^1\bar{\xi}^2-\xi^2\bar{\xi}^1>0$. In addition to that, since $\psi_2$ is positive, $\psi_2=(1-|\xi^1|^2)(1-|\xi^2|^2)$ is a superharmonic prior function.

Additionally, it is found that $\psi_3=(1-\xi^1\bar{\xi^2})(1-\xi^2\bar{\xi^1})(1-|\xi^1|^2)(1-|\xi^2|^2)$ is also a superharmonic prior function. The Laplace--Beltrami operator acting on this prior function gives 
\begin{equation}
(\frac{\Delta \psi_3}{\psi_3})=-\frac{6}{\mathcal{G}}\frac{|\xi^1-\xi^2|^2}{(1-\xi^1\bar{\xi^2})(1-\xi^2\bar{\xi^1})(1-|\xi^1|^2)(1-|\xi^2|^2)}=-6 \nonumber
\end{equation}
and it is straightforward that $\psi_3$ is superharmonic, because $\psi_3$ is positive. This prior function is similar to the prior function found in the literature \cite{Tanaka:2006, Tanaka:2009}. If the prior function is represented in the complexified coordinates, the prior function is $(1-|\xi^1|^2)$, because the two coordinates in his paper are complex conjugate to each other.

	To obtain superharmonic priors, the superharmonic prior functions found above are multiplied by the Jeffreys prior, which is the volume form of the information manifold. After that, the superharmonic priors outperform the Jeffreys prior \cite{Komaki:2006}.

\section{Conclusion}
\label{sec_kahler_conclusion} 

In this paper, we prove that the information geometry of a signal filter with a finite complex cepstrum norm is a K\"ahler manifold. The conditions on the transfer function of the filter make the Hermitian structure explicit. The first condition on the transfer function for the K\"ahlerian information manifold is whether or not multiplication between the zero-th degree terms in $z$ of the unilateral part and the analytic part in the transfer function decomposition is a constant. The second condition is whether or not the coefficient of the highest degree in $z$ is a constant in the model parameters. These two conditions are equivalent to each other for some transfer functions.

It is also found that the square of the Hardy norm of a logarithmic transfer function is the K\"ahler potential of the information geometry. It is also known as the unweighted complex cepstrum norm of a linear system. Using the K\"ahler potential, it is easy to derive the geometric objects, such as the metric tensor, the $\alpha$-connection and the Ricci tensor. Additionally, the K\"ahler potential is a constant term in $\alpha$ of the $\alpha$-divergence, {\em i.e.}, it is related to the 0-divergence.

The K\"ahlerian information geometry for signal processing is not only mathematically interesting, but also computationally practical. Contrary to non-K\"ahler manifolds where tedious and lengthy calculation is needed in order to obtain the tensors, it is relatively easier to calculate the metric tensor, the connection and the Ricci tensor on a K\"ahler manifold. Taking derivatives on the K\"ahler potential provides the metric tensor and the connection on a K\"ahler manifold. The Ricci tensor is obtained from the determinant of the metric tensor. Moreover, $\alpha$-generalization on the curvature tensor, the Ricci tensor and the scalar curvature is linear in $\alpha$. Meanwhile, there exist the non-linear corrections in the non-K\"ahler cases. Additionally, since the Laplace--Beltrami operator in K\"ahler geometry is of the simpler form, it is more straightforward to find superharmonic priors.

The information geometries of the AR, the MA and the ARMA models, the most well-known time series models, are the K\"ahler manifolds. The metric tensors, the connections and the divergences of the linear system geometries are derived from the the K\"ahler potentials with simplified calculation. In addition to that, the superharmonic priors for those models are found with much less computational efforts.

\section*{Acknowledgments}
We are thankful to Robert J. Frey and Michael Tiano for useful discussions.

\end{document}